\documentclass[12pt,leqno]{amsart} 
\usepackage{amsmath,amstext,amsthm,amssymb,amsxtra} 
\usepackage[bookmarksopen,bookmarksdepth=3,colorlinks,citecolor=red,pagebackref,hypertexnames=true]{hyperref} 
\usepackage[backrefs,msc-links,nobysame]{amsrefs}
\usepackage[normalem]{ulem}
 
\usepackage{color}
\usepackage{array}

\usepackage{txfonts} 
\usepackage[T1]{fontenc}
\usepackage{lmodern} 

\rmfamily

\usepackage{mathtools,MnSymbol} 

\allowdisplaybreaks

\theoremstyle{plain}

\theoremstyle{plain}

\newtheorem{theorem}[equation]{Theorem}
\newtheorem*{theorem*}{Theorem}

\theoremstyle{definition}
\newtheorem{definition}[equation]{Definition}

\numberwithin{equation}{section}

\def\norm#1.#2.{\lVert#1\rVert_{#2}}
\def\Norm#1.#2.{\bigl\lVert#1\bigr\rVert_{#2}}
\def\NOrm#1.#2.{\Bigl\lVert#1\Bigr\rVert_{#2}}
\def\NORm#1.#2.{\biggl\lVert#1\biggr\rVert_{#2}}
\def\NORM#1.#2.{\Biggl\lVert#1\Biggr\rVert_{#2}}

\def\R{\mathbb R}

\def\vint_#1{\mathchoice%
      {\mathop{\kern 0.2em\vrule width 0.6em height 0.69678ex depth -0.58065ex
              \kern -0.8em \intop}\nolimits_{\kern -0.4em#1}}%
      {\mathop{\kern 0.1em\vrule width 0.5em height 0.69678ex depth -0.60387ex
              \kern -0.6em \intop}\nolimits_{#1}}%
      {\mathop{\kern 0.1em\vrule width 0.5em height 0.69678ex depth -0.60387ex
              \kern -0.6em \intop}\nolimits_{#1}}%
      {\mathop{\kern 0.1em\vrule width 0.5em height 0.69678ex depth -0.60387ex
              \kern -0.6em \intop}\nolimits_{#1}}}
\def\vintslides_#1{\mathchoice%
      {\mathop{\kern 0.1em\vrule width 0.5em height 0.697ex depth -0.581ex
              \kern -0.6em \intop}\nolimits_{\kern -0.4em#1}}%
      {\mathop{\kern 0.1em\vrule width 0.3em height 0.697ex depth -0.604ex
              \kern -0.4em \intop}\nolimits_{#1}}%
      {\mathop{\kern 0.1em\vrule width 0.3em height 0.697ex depth -0.604ex
              \kern -0.4em \intop}\nolimits_{#1}}%
      {\mathop{\kern 0.1em\vrule width 0.3em height 0.697ex depth -0.604ex
              \kern -0.4em \intop}\nolimits_{#1}}}

\newcommand{\aveint}[2]{\mathchoice%
      {\mathop{\kern 0.2em\vrule width 0.6em height 0.69678ex depth -0.58065ex
              \kern -0.8em \intop}\nolimits_{\kern -0.45em#1}^{#2}}%
      {\mathop{\kern 0.1em\vrule width 0.5em height 0.69678ex depth -0.60387ex
              \kern -0.6em \intop}\nolimits_{#1}^{#2}}%
      {\mathop{\kern 0.1em\vrule width 0.5em height 0.69678ex depth -0.60387ex
              \kern -0.6em \intop}\nolimits_{#1}^{#2}}%
      {\mathop{\kern 0.1em\vrule width 0.5em height 0.69678ex depth -0.60387ex
              \kern -0.6em \intop}\nolimits_{#1}^{#2}}}

\DeclareMathOperator{\sgn}{sgn}
\DeclareMathOperator{\supp}{supp}
\DeclareMathOperator{\pv}{p.v.}

\def\ip#1,#2,{\langle #1,#2\rangle}
\def\Ip#1,#2,{\bigl\langle#1,#2\bigr\rangle}
\def\IP#1,#2,{\Bigl\langle#1,#2\Bigr\rangle}

\begin{document}
\raggedbottom
 \title[BMO through endpoint commutator bounds]{A characterization of BMO in terms of endpoint bounds for commutators of singular integrals}
 
 \author[N. Accomazzo]{Natalia Accomazzo}
\address{Departamento de Matem\'aticas, Universidad del Pais Vasco, Aptdo. 644, 48080
Bilbao, Spain}


\thanks{The author is supported by grant  MTM2014-53850 of the Ministerio de Econom\'ia y Competitividad (Spain), grant IT-641-13 of the Basque Government}
 
\begin{abstract} We provide a characterization of $\mathrm{BMO}$ in terms of endpoint boundedness of commutators of singular integrals. In particular, in one dimension, we show that $\|b\|_{\mathrm{BMO}}\eqsim B$, where $B$ is the best constant in the endpoint $L\log L$ modular estimate for the commutator $[H,b]$. We provide a similar characterization of the space $\mathrm{BMO}$ in terms of endpoint boundedness of higher order commutators of the Hilbert transform. In higher dimension we give the corresponding characterization of $\mathrm{BMO}$ in terms of the first order commutators of the Riesz transforms. We also show that these characterizations can be given in terms of commutators of more general singular integral operators of convolution type. 
\end{abstract}
\maketitle

\section{Introduction} 

The main subject of this paper is commutators of singular integrals operators with appropriate symbols. In particular we are interested in characterizing the class of symbols for which these commutators are bounded. To be more precise, let $b$ be a locally integrable function on $\mathbb{R}^n$, which we call the symbol, and $T$ a Calder\'on Zygmund operator. For smooth functions we define the commutator operator $[b,T]$ as
\[[b,T]f\coloneqq T(bf)-bT(f).\]
In 1976, Coifman, Rochberg and Weiss proved that the commutator is a bounded map from $L^p(\mathbb{R}^n)$ onto itself $(1<p<\infty)$ if the symbol belongs to $\mathrm{BMO(\mathbb{R}^n)}$, \cite{MR0412721}. They also show that if all commutators $[R_j,b]$, $1\le j \le n$, with the Riesz transforms are bounded then necessarily $b\in\mathrm{BMO(\mathbb{R}^n)}$. In \cite{MR524754}, Janson improved this result by showing that it suffices to assume the boundedness of only one of  these commutators $[R_j,b]$. Commutators of more general  singular integral operators were considered 
by Uchiyama in \cites{MR0467384,Uchiyama2} and Li in \cite{MR1373281}. In the multiparameter case the corresponding results have been the subject of a long investigation and similar characterizations of $\mathrm{BMO(\mathbb{R}^n)}$ are available both for small $\mathrm{BMO(\mathbb{R}^n)}$, as well as for product $\mathrm{BMO(\mathbb{R}^n)}$, by means of suitable commutator bounds; see for example \cites{MR1785283,FergLacey,OuPetermichlStrouse}.

For the case $p=1$, the endpoint boundedness of commutators of singular integrals with $\mathrm{BMO(\mathbb{R}^n)}$ symbols was addressed by P\'erez in \cite{MR1317714}. More specifically he showed that the commutator of a singular integral with a $\mathrm{BMO(\mathbb{R}^n)}$ symbol is not bounded from $L^1(\mathbb{R}^n)$ onto $L^{1, \infty}(\mathbb{R}^n)$. However, they do satisfy the following modular inequality of the type $L\log L$ when $b \in \mathrm{BMO(\mathbb{R}^n)}$
\begin{equation}\label{LlogLendpoint}
|\{x \in \mathbb{R}^n:\, |[b,T]f(x)|>t\}|\le \int_{\mathbb{R}^n} \frac{\|b\|_{\mathrm{BMO(\mathbb{R}^n)}}|f(x)|}{t}\left(1+\log^+\left(\frac{\|b\|_{\mathrm{BMO(\mathbb{R}^n)}}|f(x)|}{t}\right)\right) dx.
\end{equation}
 This estimate reflects that these commutators are more singular operators than Calder\'on Zygmund operators.   
Estimates like these are interesting since they serve as endpoint to interpolate. The original proof of \eqref{LlogLendpoint} is based on a good-$\lambda$ type argument relating 
these commutators with $M^2=M\circ M$, the iteration of the maximal function. Actually $M^2$ satisfies a similar  $L\log L$ modular inequality. 
A different proof of this Theorem was given by P\'erez and Pradolini in \cite{MR1827073}. 

 The main purpose of this work is to show the necessity of the $\mathrm{BMO(\mathbb{R}^n)}$ assumption for the boundedness of commutators of singular integrals at the endpoint. First, we investigate the simplest case, namely if we assume that the commutator with the Hilbert transform satisfies the modular $L\log L$ endpoint inequality then the symbol $b$ has to be a $\mathrm{BMO(\mathbb{R})}$ function. We obtain the following theorem: 
\begin{theorem}\label{HT}
Let $b \in L^1_{\mathrm{loc}}(\mathbb{R})$ and $H$ be the Hilbert transform. If we have that there exists a constant $B$ such that
\[|\{x \in \mathbb{R}:\, |[H,b]f(x)|>t\}|\le \int_{\mathbb{R}} \frac{B|f(x)|}{t}\left(1+\log^+\left(\frac{B|f(x)|}{t}\right)\right) dx\]
for all $t>0$ and $f$, then $b \in \mathrm{BMO(\mathbb{R})}$ and there is a universal constant $c$ such that $\|b\|_{\mathrm{BMO(\mathbb{R})}}\leq c B$. 
\end{theorem}
Here and throughout the paper we define $\log^+t\coloneqq \max(\log t,0)$.

We will not prove this theorem, as it will follow from the corresponding result for  higher order commutators of the Hilbert transform included in Section \ref{dimOne}.

In Section \ref{HighDim} we will address the same question but in higher dimensions. We will consider more general singular integral operators, 
\[Tf(x)\coloneqq \pv \int_{\mathbb{R}^n} \frac{\Omega(x-y)}{|x-y|^n} f(y) dy,\]
where we will impose conditions on the kernel $\Omega$ similar as in the paper of Uchiyama, \cite{MR0467384}. In particular, we obtain the same result assuming that the commutator with one Riesz transform satisfies the $L \log L$ endpoint inequality.

As we mentioned before, there are examples of commutators of Calder\'on Zygmund operators and BMO functions that fail to be of weak type $(1,1)$. However, it is not hard to see that if we take the symbol to be an $L^{\infty}$-function instead, then we get that the commutator $[b,T]$ is a bounded map from $L^1(\mathbb{R}^n)$ to $L^{1,\infty}(\mathbb{R}^n)$. In section~\ref{HighDim} we prove that the condition $b \in L^{\infty}(\mathbb{R}^n)$ is also necessary, by assuming that the commutator $[b,T]$ of a locally integrable symbol $b$ and a singular integral operator (with the same conditions on the kernel that we considered above) is of weak type $(1,1)$.

\section{Dimension one}\label{dimOne}

\subsection{BMO}
We say that a function $b \in L^1_{\mathrm{loc}}(\mathbb{R}^n)$ belongs to the class $\mathrm{BMO(\mathbb{R}^n)}$ if 
\[\|b\|_{\mathrm{BMO(\mathbb{R}^n)}}\coloneqq \sup_Q \frac{1}{|Q|} \int_Q |b(x) -b_Q| dx < \infty\]
where the supremum is taken over all the cubes with edges parallel to the axes and $b_Q$ denotes the usual average of $b$ over $Q$, namely $b_Q=\frac{1}{|Q|}\int_Qb(x)\,dx$. In this space, we have an equivalent norm, defined by
\[\|b\|'_{\mathrm{BMO(\mathbb{R}^n)}}\coloneqq  \sup_Q \inf_c \frac{1}{|Q|}\int_Q |b(x)-c| dx.\]
For a cube $Q$, the infimum above is attained and the constants where this happens can be found among the \textit{median values}.
\begin{definition}
 A \textit{median value} of $b$ over a cube $Q$ will be any real number $m_Q(b)$ that satisfies simultaneously 
\[|\{ x \in Q: b(x)> m_b(Q)\}|\le \frac{1}{2} |Q|\]
and
\[|\{x \in Q:\, b(x)<m_b(Q)\}|\le \frac{1}{2}|Q|.\]
\end{definition}
The fact that the constant $c$ in the definition of $\|b\|'_{\mathrm{BMO(\mathbb{R}^n)}}$ can be chosen to be a median value of $b$ can be found for instance in \cite[Ch. 8, p. 199]{MR869816}.

An equivalent description of $\mathrm{BMO(\mathbb{R}^n)}$ was obtained  by John \cite{MR0190905} and by Str\"omberg \cite{MR529683}. These authors considered the following quantities for $0<s<1$ and $b$ measurable
\[\|b\|_{\mathrm{BMO}_s} \coloneqq  \sup_Q \inf_c \inf \{t\ge 0:\, |\{x \in Q:\, |b(x)-c|>t\}|\le s |Q|\}\]
and proved that $\|b\|_{\mathrm{BMO}_{s}}$ is equivalent to the usual $\mathrm{BMO(\mathbb{R}^n)}$-norm for $0<s\le 1/2$. Here we will understand that $\mathrm{BMO_s} \equiv \mathrm{BMO_s(\mathbb{R}^n)}$, we omit the dimension to simplify notation. They obtained the following more precise estimates.
\begin{theorem}[Str\"omberg, \cite{MR529683}]\label{thmStromberg}
For $0<s\le 1/2$ there exists a constant $C$ depending only on $n$ such that
\[s\|b\|_{\mathrm{BMO}_{s}}\le \|b\|_{\mathrm{BMO}}\le C\|b\|_{\mathrm{BMO}_{s}}.\]
\end{theorem}
For these ``norms" it will be also useful to replace the general constant $c$ by the median $m_Q(b)$. Thus, to prove that a function $b$ belongs to $\mathrm{BMO}$ it will be enough to find constants $A$ and $s$ ($0<s\le 1/2$) such that, for every cube $Q$ we have
\[|\{ x \in Q:\, |b(x)-m_Q(b)|>A\}|\le s |Q|.\]
Then we also have that $\|b\|_{\mathrm{BMO}}\lesssim_{n,s} A$. 

\subsection{Higher order commutators}
The \textit{commutator of order k} for $k=2,3, \dots$ is defined by the recursive formula $T^k_b\coloneqq[T_b^{k-1},b]$. For $k=1$ we define $T^1_b$ as the usual commutator $T^1_b\coloneqq[b,T]$. For $T$ a Calder\'on-Zygmund operator and $b \in \mathrm{BMO(\mathbb{R}^n)}$, we have the following estimate
\[
|\{x \in \mathbb{R}^n:\, |T^k_bf(x)|>t\}|\le \int_{\mathbb{R}^n} \phi_k\left(\frac{\|b\|_{\mathrm{BMO(\mathbb{R}^n)}}^k|f(x)|}{t}\right) dx
\]
for every smooth function with compact support $f$ and $t>0$; here the function $\phi_k$ is defined by $\phi_k(t)\coloneqq t(1+\log^+(t))^k$. 

Consider the Hilbert transform, defined by
\[Hf(x)\coloneqq \pv \int_{\mathbb{R}} \frac{f(y)}{x-y} dy.\]
It is not difficult to see that in this case we can define the $k$-order commutator of the Hilbert transform via the formula
\[H^k_b f(x)= \pv \int_{\mathbb{R}} \frac{(b(x)-b(y))^k}{x-y} f(y) dy.\]

The following theorem, in combination with the results from \cite{MR1317714}, gives a characterization of the space $\mathrm{BMO(\mathbb{R}^n)}$ in terms of endpoint boundedness of higher order commutators of the Hilbert transform.

\begin{theorem}\label{BMOnecessityHilbert}
Let $b \in L^1_{\mathrm{loc}}(\mathbb{R})$. If there exists a constant $B$ and a positive integer $k$ such that we have the following estimate
\[|\{x \in \mathbb{R}:\, |H^k_bf(x)|>t\}|\le \int_{\mathbb{R}} \phi_k\left(\frac{B|f(x)|}{t}\right) dx,\]
then $b \in \mathrm{BMO(\mathbb{R})}$ and $\|b\|_{\mathrm{BMO(\mathbb{R})}} \lesssim B^{1/k}$. 
\end{theorem}

\begin{proof} 
By \ref{thmStromberg}, it is enough to find a constant $A$ such that for every interval $I$,
\[|\{x \in I:\, |b(x)-m_I(b)|^k>A\}| \le \frac{1}{2} |I|,\]
where $m_I(b)$ is a median of $b$ on $I$. 

Fix $I$ an interval. We can find disjoint subsets of $I$, $E_+$ and $E_-$ such that $|E_+|=|E_-|=|I|/2$, 
\[\begin{split}
E_+\subset \{y \in I:\, b(y)\ge m_I(b)\}\\
E_-\subset \{y \in I:\, b(y)\le m_I(b)\}.
\end{split}\]
Then,\[|b(x)-m_I(b)|^k\chi_I(x) = (b(x)-m_I(b))^k \chi_{E_+}(x) + (m_I(b)-b(x))^k \chi_{E_-}(x).\]
For $y \in E_-$ and $z \in E_+$ we have
\[|b(x)-m_I(b)|^k\chi_I(x) \le (b(x)-b(y))^k \chi_{E_+}(x) + (b(z)-b(x))^k \chi_{E_-}(x).\]
Now integrating for $y \in E_-$ and $z \in E_+$ and calling $c_I$ the center of $I$, we get
\begin{flalign*}
&|b(x)-m_I(b)|^k \chi_I(x)\le \frac{1}{|E_-|}\int_{E_-} (b(x)-b(y))^k \chi_{E_+}(x) dy + \frac{1}{|E_+|}\int_{E_+}  (b(x)-b(z))^k \chi_{E_-}(x) dz.
\end{flalign*}
The first summand in the right hand side of the estimate above can be bounded above by
 \[
 \begin{split}
  \frac{1}{|E_-|}\int_{E_-}(b(x)-b(y))^k \chi_{E_+}(x) dy &\leq \frac{1}{|E_-|}\int_{\mathbb{R}} \frac{(b(x)-b(y))^k}{x-y}(x-c_I)\chi_{E_+}(x)\chi_{E_-}(y) dy 
   \\
 &\qquad +\frac{1}{|E_-|} \int_{\mathbb{R}} \frac{(b(x)-b(y))^k}{x-y}(c_I-y)\chi_{E_+}(x)\chi_{E_-}(y)  dy  
 \\
 &\leq 2\frac{|x-c_I|}{|I|} |H^k(\chi_{E_+})(x)|  +  2\left|H^k\left(\frac{(\cdot - c_I)}{|I|}\chi_{E_-}\right)(x)\right|.
   \end{split}
\]
Using a similar estimate for the second summand we get
\[\begin{split}
&|\{x \in I:\, |b(x)-m_I(b)|^k> A\}|\le |\{x \in \mathbb{R}:\, |H^k(\chi_{E_+})(x)|>A/8\}|\\
&+|\{x \in \mathbb{R}:\, |H^k((\cdot - c_I)/|I|\chi_{E_-}(x)>A/8\}|\\
&+|\{x \in \mathbb{R}:\, |H^k(\chi_{E_-})(x)|>A/8\}|+ |\{x \in \mathbb{R}:\, |H^k((\cdot - c_I)/|I|\chi_{E_+})(x)|>A/8\}|\\
&\eqqcolon  (i)+(ii)+(iii)+(iv).
\end{split}\]
We show the estimate for $(i)$. The estimates for the other terms are similar. 
\[
(i)\le \int_{\mathbb{R}} \chi_{E_+}(x)\frac{8B}{A} \left(1+\log^+\left(\chi_{E_+}(x)\frac{8B}{A}\right)\right)^k dx\le |E_+|\frac{1}{4}=\frac{|I|}{8},\]
if we choose $A=32B$. 
Summing, 
\[|\{x \in I:\, |b(x)-m_I(b)|^k>A\}| \le \frac{1}{2} |I|\]
as we wanted.\end{proof}
\section{Higher dimensions}\label{HighDim}
For this section we will be considering operators $T$ of the form
\begin{equation}\label{eq:generalsi}
	Tf(x)\coloneqq \pv \int_{\mathbb{R}^n} \frac{\Omega(x-y)}{|x-y|^n}f(y) dy,
\end{equation}
where $\Omega\in \mathrm{Lip}(S^{n-1})$ is homogeneous of degree zero, satisfies $\int_{S^{n-1}} \Omega =0$, and the set $\{\Omega (x)= 0\}$ has zero measure. An important class of operators that satisfies these conditions are the Riesz transforms, 
\[R_jf(x) \coloneqq \pv \int_{\mathbb{R}^n} \frac{x_j-y_j}{|x-y|^{n+1}} f(y) dy.\]

First we prove that if the commutator of one of these operators with a symbol $b$ is weak $(1,1)$ then $b$ must satisfy a stronger condition than $\mathrm{BMO(\mathbb{R}^n)}$, namely $b \in L^{\infty}(\mathbb{R}^n)$. 
\begin{theorem}\label{weak(1,1)}
Let $b$ be a locally integrable function and suppose $[b,T]:\, L^1(\mathbb{R}^n) \to L^{1, \infty}(\mathbb{R}^n)$ is bounded. Then $b \in L^{\infty}(\mathbb{R}^n)$ and we have  the bound $\|b\|_{\infty} \le C(\Omega, n)\|[b,T]\|_{L^1 \to L^{1,\infty}}$.  
\end{theorem}
\begin{proof} We begin by fixing a locally integrable function $b$. Note that this assumption implies that $b$ is finite almost everywhere, and that almost every point $y\in\R^n$ is a Lebesgue point of $b$.
Now recall that
\[ [b,T]f(x)= \pv \int_{\mathbb{R}^n} \frac{b(x)-b(y)}{|x-y|^n} \Omega(x-y) f(y) dy.\]
By renormalizing $b$, we can assume $\|[b,T]\|_{L^1 \to L^{1, \infty}}=1$.Take $f$ to be a $C^{\infty}$ function with compact support, even, $\supp f \subset B(0,1)$, $\int f =1$, and $0\le f \le 1$. For every $\varepsilon >0$, set $f_{\varepsilon} (x)\coloneqq \frac{1}{\varepsilon^n} f\left(\frac{x}{\varepsilon}\right)$ and $f^y_{\varepsilon}(x)\coloneqq f_{\varepsilon}(y-x)$. Then, whenever $y$ is a Lebesgue point of $b$, we have
\[\lim_{\varepsilon \to 0} |[b,T]f^y_{\varepsilon} (x)| = \frac{|b(x)-b(y)|}{|x-y|^n}|\Omega(x-y)|.\]
So we get that, for every $\lambda>0$ and $y$ a Lebesgue point for $b$, 
\begin{equation}\label{eqweak}
|\{x \in\mathbb{R}^n:\, \frac{|b(x)-b(y)|}{|x-y|^n}|\Omega(x-y)|>\lambda\}|\le \frac{\|[b,T]\|_{L^1 \to L^{1,\infty}}}{\lambda}.
\end{equation}
Fix $\varepsilon>0$ and take $K$ to be a compact subset of $S^{n-1}$ such that $\{x \in S^{n-1}:\, \Omega(x)=0\} \cap K= \emptyset$ and $|S^{n-1}\setminus K|<\varepsilon$. Call $C_{\Omega}\coloneqq \inf \{|\Omega(x)|:\, x \in K\}$ and note that $C_\Omega>0$, by the Lipschitz assumption on $\Omega$. We now define the following sets
\[\begin{split}
&\Lambda_{\lambda}(y)\coloneqq\{x \in \mathbb{R}^n:\, \frac{x-y}{|x-y|} \in K, \frac{|b(x)-b(y)|}{|x-y|^n}>\lambda\},
\\
& S_K(y)\coloneqq \{x \in \mathbb{R}^n:\, \frac{x-y}{|x-y|} \in K\}.
\end{split}\]
Note that, with the definition above, and the choice of $K$, we have for every $r>0$ that 
\begin{equation}\label{eq:smallmeasure}
|B(0,r)\setminus S_K(0)|<\epsilon r^n/n.
\end{equation}

By \ref{eqweak} and the definition of $C_{\Omega}$ we have
\[|\Lambda_{\lambda}(y)|\le \frac{1}{C_{\Omega}\lambda}.\]
Since our hypothesis is invariant under replacing $b$ by $b-c$, for any constant $c$, and since $b$ is finite almost everywhere, we can assume that $b(0)=0$ and we also have
\[|\Lambda_{\lambda}(0)|=\Big|\Big\{x \in \mathbb{R}^n:\, \frac{x}{|x|} \in K, \frac{|b(x)|}{|x|^n}>\lambda\Big\}\Big|\le \frac{1}{C_{\Omega}\lambda}.\]
Let $y \neq 0$ and $x \notin \Lambda_{1/|y|^n}(y)$, $x \in B(y, |y|1/2|b(y)|^{1/n})\cap S_K(y)$ 
\[\begin{split}
|b(x)|&\ge |b(y)|- \frac{|b(x)-b(y)|}{|x-y|^n}|x-y|^n \ge |b(y)|- \frac{1}{|y|^n}\big(|y|\frac{1}{2}|b(y)|^{1/n}\big)^n\\
	  &\ge \big(1-\frac{1}{2^{n}}\big)|b(y)|= c_n|b(y)|
\end{split}\]
for almost every $y$. Suppose that $|b(y)|>2^n$ (if we had that $|b(y)|\le 2^n$ for all $y \neq 0$ that is also a Lebesgue point then we would be done). We conclude that 
\begin{equation}\label{eq1}
\begin{split}
A(y,K)&\coloneqq |\{ x\in [B(y, |y|1/2|b(y)|^{1/n})\cap S_K(y)\cap S_K(0)] \backslash \Lambda_{1/|y|^n}(y):\, \frac{|b(y)|}{|x|^n}>\frac{1}{|y|^n}\}|
\\ 
&\le |\{x \in S_K(0):\, \frac{c_n|b(x)|}{|x|^n}> \frac{1}{|y|^n}\}|= |\Lambda_{1/(c_n|y|^n)}(0)|\le C_{\Omega}^{-1}c_n |y|^n.
\end{split}
\end{equation}
Since $|b(y)|>2^n$,
\[ |y||b(y)|^{1/n} = \frac{1}{2} |y||b(y)|^{1/n} + \frac{1}{2} |y||b(y)|^{1/n} \ge |y| + \frac{1}{2}|y||b(y)|^{1/n}.\]
This implies that $B(y, 1/2|y||b(y)|^{1/n}) \subset B(0, |y||b(y)|^{1/n})$ and so
\begin{equation}\label{eq2}
|A(y,K)| \ge|B(y, |y|1/2|b(y)|^{1/n})\cap S_K(y)|-|B(0,|y||b(y)|^{1/n})\cap S_K(0)^c|
-| \Lambda_{1/|y|^n}(y)|. 
\end{equation}
Let us observe that
\[\begin{split}
&|B(y, |y|1/2|b(y)|^{1/n})\cap S_K(y)|= |B(0, 1/2|y||b(y)|^{1/n})\cap S_K(0)|\\
&\quad =\frac{1}{2^n}|y|^n|b(y)|(|B(0,1)|-|B(0,1)\cap S_K(0)^c|)\ge \frac{1}{2^n}|y|^n|b(y)| \left(\omega_n-\frac{\varepsilon}{n}\right)
\end{split}\]
by \eqref{eq:smallmeasure}; here $\omega_n$ denotes the measure of the unit ball. 
We also have that
\[\begin{split}
|B(0,|y||b(y)|^{1/n})\cap S_K(0)^c|+| \Lambda_{1/|y|^n}(y)|&\le|y|^n|b(y)||B(0, 1)\cap S_K(0)^c|+C_{\Omega}^{-1}|y|^n\\
&\le |y|^n|b(y)| \frac{\varepsilon}{n} + C_{\Omega}^{-1}|y|^n,
\end{split}\]
here we are using \eqref{eq:smallmeasure} again. Estimate \eqref{eq2} then yields
\[
|A(y,K)| \ge |y|^n|b(y)|\left(\frac{1}{2^n}\omega_n-\frac{\varepsilon}{n}\frac{2^n+1}{2^n}\right)-C_{\Omega}^{-1}|y|^n.
\] 
Now take $\varepsilon=\frac{n}{2}\omega_n\frac{1}{2^n+1}$. Combining with the previous estimate we get that
\[c_n C|y|^n \ge \frac{\omega_n}{2^{n+1}}|y|^n|b(y)|-C_{\Omega}^{-1}|y|^n\]
and so
\[|b(y)| \le \max \Big\{2^n, \frac{2^{n+1}(c_n+1)}{C_{\Omega}\omega_n}\Big\}=\max \Big\{2^n, \frac{2}{\omega_n}(2^{n+1}-1) \frac{1}{C_{\Omega}}\Big\}\eqqcolon C(\Omega,n)\]
for almost all $y \in \mathbb{R}^n$ and thus $b$ is bounded and $\|b\|_\infty\leq C(\Omega,n)$ as desired. 
\end{proof}
Now we prove a higher dimensional analogue of  Theorem \ref{HT} for the class of singular integral operators given in \eqref{eq:generalsi}. We use a similar argument as the one given by Uchiyama in \cite{MR0467384}. In the statement and proof of the following theorem we test some endpoint inequalities for the commutator on characteristic functions, which is somewhat reminiscent of Sawyer's testing conditions, \cite{Sawyer}, characterizing the two-weight norm inequalities for the Hardy-Littlewood maximal operator.  We impose a symmetric condition on the adjoint operator; indeed, since we are assuming an endpoint estimate, we can no longer rely on duality in order to conclude the boundedness of the adjoint commutator, as for example was done in \cite{MR0467384}. Note however that for the Riesz transforms, as well as for more general odd kernels as in \eqref{eq:generalsi}, it will be enough to assume the endpoint boundedness of $[b,T]$ at the endpoint in order to conclude that $b \in \mathrm{BMO(\mathbb{R}^n)}$ (we would get the condition on the adjoint for free, since $[b,T]^*=[b,T^*]=[b,-T]$ for odd convolution kernels). 

For the statement of the theorem below we remember that $\phi_1(t)=t(1+\log^+t)$.
\begin{theorem}
Let $b$ be a locally integrable function on $\mathbb{R}^n$. If there exists a constant $B$ such that for every measurable set $E$ and $t>0$ we have that
\[|\{x \in \mathbb{R}^n:\, |[b,T]\chi_E(x)|>t\}| \le \int_{\mathbb{R}^n} \phi_1\left(\frac{B\chi_E(x)}{t}\right) dx,\]
and 
\[|\{x \in \mathbb{R}^n:\, |[b,T^*]\chi_E(x)|>t\}| \le \int_{\mathbb{R}^n} \phi_1\left(\frac{B\chi_E(x)}{t}\right) dx,\]
then $b \in \mathrm{BMO(\mathbb{R}^n)}$ and $\|b\|_{\mathrm{BMO(\mathbb{R}^n)}}\le C(\Omega, n) B$. 
\end{theorem}
\begin{proof}
As we did in the proof of theorem \ref{weak(1,1)}, we can assume $B=1$. Define $M(b,Q)\coloneqq \inf\limits_{c \in \mathbb{R}} \frac{1}{|Q|} \int_{Q} |b(y)-c| dy$. We want to prove 
\begin{equation}\label{eqM}
 \sup_Q M(b,Q) \le C(\Omega, n).\end{equation}
By translation and dilation invariance it suffices to prove \eqref{eqM} for the cube $Q_1= \{x \in \mathbb{R}^n:\, |x| < (2\sqrt n)^{-1}\}. $ 

Let $M\coloneqq M(b, Q_1)=|Q_1|^{-1} \int_{Q_1} |b(y)-m_{Q_1}(b)| dy$, where $m_{Q_1}(b)$ is a median of $b$ over $Q_1$. Since $[T,b-m_{Q_1}(b)]=[T, b]$ we may assume that $m_{Q_1}(b)=0$. This means that we can find disjoint subsets of $Q_1$, $E_1 \supset  \{ x \in Q_1:\, b(x)<0\}$ and $E_2 \supset \{x \in Q_1:\, b(x)>0\}$ of equal measure. Define $\psi\coloneqq  \chi_{E_2}-\chi_{E_1}$.  
Then $\psi$ satisfies: $\|\psi\|_{\infty}=1,\, \supp \psi \subset Q_1,$
\[\begin{split}
\int \psi (x) dx=0,  \quad
 \psi(x)b(x) \ge 0, \quad
 \text{and} \quad |Q_1|^{-1} \int \psi(x)b(x) dx=M.
\end{split}\]
Take $\Sigma \subset S^{n-1}$ a compact set such that $\Omega(x)>0$ for every $x \in \Sigma$. From now on, we will denote by $A_i$  constants depending only on the dimension $n$ and the kernel $\Omega$. Take $A_1$ such that for every $x \in \Sigma$ and $z \in S^{n-1}$ satisfying $|x-z|< A_1$, we have $|\Omega(x)-\Omega(z)| < 1/2 \Omega(x)$. Denote $x'=x/|x|$. Then, for $x \in G\coloneqq \{x \in \mathbb{R}^n:\, |x|>A_2=2A_1^{-1}+1 \text{ and } x' \in \Sigma\}$, 
\[\begin{split}
|[b,T]\psi(x)|&=|T(b\psi)(x)-b(x)T\psi(x)|\ge |T(b\psi)(x)|-|b(x)||T\psi(x)|.
		      \end{split}\]
We bound these two terms separately. We deal with the first one, 
\[|T(b\psi)(x)|=\left|\pv \int_ {Q_1} \frac{\Omega(x-y)}{|x-y|^{n} }b(y)\psi(y) dy\right|.\]
Observe that $|(x-y)'-x'|<A_1$ and so $\Omega(x-y)>1/2 \Omega(x)$, which in particular means that $\Omega(x-y)$ is positive. Since we already have that $b(y)\psi(y)$ is nonnegative and we are taking $x \in G$, we get
\[
\left|\pv \int_{Q_1} \frac{\Omega(x-y)}{|x-y|^{n} }b(y)\psi(y) dy\right| =  \int_{Q_1} \frac{\Omega(x-y)}{|x- y|^n} |b(y)| dy\ge A_3 M |x|^{-n}.
\]
Now we have to deal with $|T\psi(x)|$. Since we have that $\int \psi =0$ we can estimate
\[\begin{split}
\left| \pv \int_{Q_1} \frac{\Omega(x-y)}{|x-y|^{n}} \psi(y) dy\right|&= \left|  \int_ {Q_1} \left(\frac{\Omega(x-y)}{|x-y|^n} - \frac{\Omega(x)}{|x|^n} \right)\psi(y) dy\right| \\
&\le  \int_{Q_1}\left|\frac{\Omega(x-y)}{|x-y|^n}-\frac{\Omega(x)}{|x|^n}\right|  dy\\
&\le A_4 |x|^{-n-1}.
\end{split}\]
Then, we have
\[|[b,T]\psi(x)| \ge A_3 M |x|^{-n}-A_4 |b(x)||x|^{-n-1}.\]
Letting $F\coloneqq \{ x\in G:\, |b(x)|>(MA_3/2A_4)|x| \text{ and } |x|<M^{1/n}\}$ we have
\[\begin{split}
|&\{ x \in \mathbb{R}^n:\, |[b,T]\psi(x)| >A_3/2\}| \ge |\{ x \in (G\backslash F) \cap \{|x|<M^{1/n}\}:\, |[b,T]\psi(x)| >A_3/2\}|\\
 &\ge |\{ x \in (G\backslash F) \cap \{|x|<M^{1/n}\}:\, 2^{-1}A_3 M |x|^{-n} >A_3/2\}|\\
 &=|(G\backslash F)\cap \{|x|< M^{1/n}\}|=A_5(M-A_2^n)-|F|
 \end{split}\]
By our assumption, we have that
\[\begin{split}
|\{ x \in \mathbb{R}^n:\, |[b,T]\psi(x)| >A_3/2\}| \le \int_{Q_1} \phi(2A_3^{-1}|\psi(x)|) dx\le |Q_1|\phi(2A_3^{-1}).
\end{split}\]
Then
\[|F|\ge A_5(M-A_2^n)-\phi(2A_3^{-1})|Q_1|\ge A_5M/2\]
by assuming, as we may, that $M$ is large enough.

Let $g(x)\coloneqq \sgn(b(x))\chi_F(x)$ and $T^*$ be the adjoint operator of $T$. Then, for $x \in Q_1$, 
\[|[T^*,b]g(x)|\ge  |T^*(bg)(x)|-|b(x)||T^*(g)(x)|.\]
By the definition of $F$, we have
\[\begin{split}
|T^*(bg)(x)|&=\Big|\pv \int_{\mathbb{R}^n} \Omega(y-x)|x-y|^{-n} b(y)g(y)dy\Big|
\\
			&= \int_F \Omega(y-x)|x-y|^{-n} |b(y)| dy.
\end{split}\]
Note that $y \in F$ means that $|y|\le M^{1/n}$ and thus
\[\begin{split}
 |T^*(bg)(x)|& \ge A_6  \int_F \frac{MA_3}{2A_4} |y|^{-n+1} dy\\
			 & \ge A_6A_3(2A_4)^{-1}A_2 M^{1/n}|F|\ge A_7 M^{1+1/n}.
\end{split}\]
For the second summand in the estimate for $[T^*,b]$ we have for $x\in Q_1$ that
\[\begin{split}
|T^*g(x)|&\le \Big|\pv \int_F \Omega(y-x)|x-y|^{-n} g(y) dy\Big| \\
						&\le \int_F |\Omega(y-x)| |x-y|^{-n} dy\\
						&\leq \|\Omega\|_{L^\infty(S^{n-1})} \int_F |y-x|^{-n} dy
						\\
						&\le  \|\Omega\|_{L^\infty(S^{n-1})} \int_{A_2 \le |y|\le M^{1/n}} \frac{1}{|y|^{n}-2^{-n}} \le A_8 \log M.
\end{split}\]
Then, for $x \in Q_1$, \[|[T^*,b]g(x)|\ge A_7 M^{1+1/n}-A_8 |b(x)|\log M.\] 
By our assumption on $T^*$ we can now conclude that
\[\begin{split}
&|\{x \in \mathbb{R}^n:\, |[T^*,b] g(x)|\ge (A_7 /2) M^{1+1/n}\}|\le \int_{\mathbb{R}^n} \phi\left(\frac{|g(x)|}{(A_7/2)M^{1+1/n}}\right) dx \\
&=\int_F \phi([(A_7/2)M^{1+1/n}]^{-1}) dx =|F|\phi(A_9M^{-1/n-1})\\
&\le M \phi(A_9M^{-1/n-1})=A_9M^{-1/n},
\end{split}\]
where the last inequality follows by taking $M$ large enough, since $\log^+ t$ vanishes for $|t|<1$.  On the other hand, 
\[\begin{split}
|&\{x \in \mathbb{R}^n:\, [T^*,b]g(x)|\ge A_7M^{1+1/n}\}| \ge |\{x \in Q_1:\, |[T^*,b]g(x)|\ge A_7/2M^{1+1/n}\}|\\
											   & \ge |\{x \in Q_1:\, A_7M^{1+1/n}-A_8\log M|b(x)|\ge A_7/2M^{1+1/n}\}|\\
											   &= |\{x \in Q_1:\, |b(x)|\le A_{10}M^{1+1/n}(\log M)^{-1}\}|\\
											   &=|Q_1|-|\{x \in Q_1:\, |b(x)|>A_{10}M^{1+1/n}(\log M)^{-1}\}|\\
											   &\geq |Q_1|-A_{10}|Q_1|\log M M^{-1-1/n} |Q_1|^{-1}\int_{Q_1} |b(x)| dx\\
											   &=|Q_1|-A_{10} |Q_1|\log M M^{-1/n}\ge A_{11},
											  \end{split}\]
as $M^{-1/n}\log M $ is bounded for every $M> e^{1/n}$. Then, we have that
\[M\le (A_9/A_{11})^n.\]
Summarizing the estimates above, we have proved that
\[M \le \max\Big\{\frac{2}{A_5}(1+A_5A_2^n), A_9^{-1/n-1}, e^{1/n}, (A_9/A_{11})^n\Big\} \eqqcolon C(\Omega, n),\]
as all the constants $A_i$ depend only on $\Omega$ and $n$. 
\end{proof}

\section{Acknowledgments}

This work is part of the author's PhD Thesis at the University of the Basque Country.  Right before the submission of the current paper we discovered that some of our results were independently discovered in \cite{Guoetal} and with different method of proof. 

The author would like to thank the anonymous referee for insightful comments that improved the presentation.


\end{document}